\documentclass{amsart}

\usepackage{amssymb}
\usepackage{amsfonts}
\usepackage{amsthm} 

\input xy
\xyoption{all}

\usepackage{natbib}

\newtheorem{theorem}{Theorem}[section]

\newtheorem{definition}[theorem]{Definition}
\newtheorem{lemma}[theorem]{Lemma}

\newcommand{\marrow}[3]{\ensuremath{#1\colon\!#2\!\rightarrowtail\!#3}}

\newcommand{\clmarrow}[2]{\ensuremath{#1\!\rightarrowtail\!#2}}
\newcommand{\eparrow}[3]{\ensuremath{#1\colon\!#2\!\twoheadrightarrow\!#3}}
\newcommand{\cleparrow}[2]{\ensuremath{#1\!\twoheadrightarrow\!#2}}

\newcommand{\cprod}[2]{\ensuremath{#1\!\times\! #2}} 

\newdir{ >}{{}*!/-10pt/@{>}} 

\mathchardef\mhyphen="2D

\title[Set theory in higher order arithmetic]{Interpreting set theory in higher order arithmetic}

\begin{document}

\begin{abstract}A folk theorem says higher order arithmetic has the proof theoretic strength of Zermelo Frankel with limited power set.   This paper proves the theorem for two versions of power set, plus global well-ordering, and V=L.\end{abstract}

\maketitle

A folk theorem says $n$-th order arithmetic $Z_n$  has the proof theoretic strength of Zermelo Frankel set theory with restricted power set.  No precise statement has been published beyond the level of $Z_2$ and $\mathrm{ZF}$ with no power set axiom.   This paper describes and proves several versions.  First $\mathrm{ZF}[n]$ is $\mathrm{ZF}$ without the power set axiom but positing $n$ successive power sets of $\omega$: 
   \[ \beth_0=\omega \qquad \mbox{and} \qquad \beth_{i+1}=\mathcal{P}(\beth_i)\]
A proper extension $\mathrm{ZF}[n^+]$ says every set has a set of all its subsets smaller than $\beth_{n}$.  This reduces  appeals to power sets in applications. The main argument proves $Z_{n+2}$  interprets $\mathrm{ZF}[n]+(\mathrm{V=L})$ so all these theories are inter-interpretable. 

\section{The set theories  $\mathrm{ZF}[n]$, $\mathrm{ZF}[n^+]$, and $\mathrm{ZFG}[n^+]$}
The set theory $\mathrm{ZF}[0]$, often called ZF$-$, is Zermelo-Frankel without power set.  The axioms are: 

\begin{itemize}
    \item Extensionality: $ \forall z (z \in x \leftrightarrow z \in y) \rightarrow x = y $.
    \item Empty set: $\exists y \forall z \lnot [ z  \in y ]$.
    \item Regularity: $\exists a ( a \in x) \rightarrow \exists y ( y \in x \land \lnot \exists z (z \in y \land z \in x))$.
    \item Pair set: $\exists w \forall z [ z \in w \leftrightarrow (z=x \lor z=y) ]$.
    \item Sum set: $\exists u \,  \forall z [ z\in u \leftrightarrow \exists  y\,(z \in y \land y \in x) ]$.
    \item Infinity: 
               $ \exists x [\varnothing \in x \land \forall y (y \in x \rightarrow y\cup \{y\}  \in x)]$.
    \item Replacement:  For any formula $\phi(x,y) $ in the language of ZF:
                       \[  \forall x ( x\in w \rightarrow \exists! y\ \phi(x,y) ) \rightarrow \exists v\, 
                               \forall y \bigl(y\in v \leftrightarrow \exists x\, (x\in w \land \phi(x,y))\]
\end{itemize}

Examples show how replacement does some things commonly done by power set. 

\begin{theorem}  \emph{(In $\mathrm{ZF}[0]$)}  Any two sets have a cartesian product \cprod{A}{B}.
\end{theorem}
\begin{proof}  Form Kuratowski ordered pairs by repeated pair sets.   For any sets $a,B$, replace each $b\in B$ by $\langle a,b\rangle$ to get \cprod{\{a\}}{B}.  Then replace each $a\in A$ by $\cprod{\{a\}}{B}$ to get  $ \{\cprod{\{a\}}{B}\ |\ a\in A\}$ with sum set $\cprod{A}{B}=\{\langle a,b\rangle|\ a\in A,b\in B\}$.
\end{proof}

\begin{theorem}\label{T:eqrelzf}  \emph{(In $\mathrm{ZF}[0]$)}  Every equivalence relation $R\subseteq \cprod{A}{A}$ has a quotient.
\end{theorem}
\begin{proof}  Replace each $a\in A$ by its set of relata $R_a=\{a'\in A|\ \langle a,a'\rangle\in R\}$ to get the set of equivalence classes $ A/R =  \{\pi \subseteq A|\ \exists a\in A\ (\pi = R_a)\,\}$.
\end{proof}

The usual proof works in $\mathrm{ZF}[0]$ to show there is a unique set $\omega$ satisfying infinity plus induction, call it the set of numbers. 
\begin{theorem}\label{T:finite}  \emph{(In $\mathrm{ZF}[0]$)}  Every set $A$ has a set $A^{<\omega}$ of all finite strings and a set $\mathrm{Fin}(A)$ of all finite subsets of $A$.
\end{theorem}
\begin{proof} Replacing each $n\in \omega$ by the product set $A^n$ gives $\{A^n|\ n\in \omega\}$ with sum set 
 $A^{<\omega}$\!.  Replacing each $n$-tuple by the set of its entries gives  $\mathrm{Fin}(A)$.  
\end{proof}

The $\mathrm{ZF}[1]$ axioms posit $\omega$ has a power set $\mathcal{P}(\omega)$, while $\mathrm{ZF}[2]$ posits a power set $\mathcal{PP}(\omega)$, and so on for all $\mathrm{ZF}[n]$.  Clearly $\mathrm{ZF}[n]$ interprets $n+2$ order arithmetic~$Z_{n+2}$.  Then $\mathrm{ZF}[n^+]$ is $\mathrm{ZF}[0]$ plus an axiom saying every set $A$ has  a set of all subsets smaller than $\beth_{n}$.  That is, a set of all $S\subseteq A$ with a one-to-one function \clmarrow{S}{\beth_{n}} and no bijection.  If  $\mathrm{ZFC}$ is consistent then $\mathrm{ZF}[n]$ does not imply $\mathrm{ZF}[n^+]$.

\begin{theorem}\label{T:properext} $\mathrm{ZFC}$ proves no $\mathrm{ZF}[n]$ even implies $\mathrm{ZF}[1^+]$.
\end{theorem}
\begin{proof} In $\mathrm{ZFC}$ the set of sets hereditarily of cardinality $\leq \aleph_{\omega}$\! models every $\mathrm{ZF}[n]$ while $\aleph_{\omega}$\! has more than $\aleph_{\omega}$\! countable subsets by  K\"onig's inequality.
\end{proof}

Let $\mathrm{ZFG}[n^+]$ be $\mathrm{ZF}[n^+]$ plus global well-ordering.  It posits a linear order $\leq_{\gamma}$ on sets, where every proper initial segment of $\leq_{\gamma}$ is a well ordered set, and replacement allows formulas with  $\leq_{\gamma}$. 

Standard constructibility arguments as in \citet[Chapter III]{CohenSet} work in $\mathrm{ZF}[0]$ to show $\mathrm{L}$ verifies global choice plus GCH in the sense that if $\aleph_{m}$ exists it is $\beth_{m}$.  But we will rather show $Z_{n+2}$ interprets $\mathrm{ZFG}[n^+]$ plus $\mathrm{V=L}$.  

\section{Requirements from higher order arithmetic}

\subsection{Basics }\label{S:basics}
Our $n$-th order arithmetic $Z_n$ uses successively higher types but no product types.  So first order arithmetic PA$=Z_1$  has  number terms.  Second order $Z_2$ adds second order terms for classes of numbers.   We write  $\forall_i$ or $\exists_i$ to quantify over $i$-th order variables.  We adopt the first order Peano axioms except that $Z_2$ and above state induction with a second order variable:
    \[ \forall_2 X\ [ (0\in X \land  \forall_1 y (y \in X \rightarrow S(y)  \in X)  
                             \rightarrow \forall_1 y\ ( y\in X)\,]\]

We adopt extensionality, and full comprehension for each order:
\begin{align*}  \forall_{i-1}x\ (\,x\in X \leftrightarrow x\in Y)  \rightarrow X=Y  & \\
                 \exists_{i+1} X\  \forall_i x&\ (x\in X \leftrightarrow \phi(x)\,)
\end{align*}
for any formula $\phi(x)$ with $X$ not free.   Compare \citet[p.~4]{SimpsonBook} extended to higher order, or \citet[pp.~192]{TakeutiProof} with extensionality but no product types.

\subsection{Sequences of classes}\label{SS:sequences}

Take any coding of ordered pairs $\langle j,k\rangle$ of numbers by numbers.  For each order $i$ define an indexing operator $\star_i$:
  \[  \forall_1 j,k\ ( j\star_1\! k = \langle j,k\rangle) \qquad  \forall_i X\, \forall_1 j\ ( j\star_i\! X = \{j\star_{i-1}\! Y\  | \ Y\in X\})\]
Think of $\langle j,k \rangle$ as the number $k$ with index $j$.  Then $j\star_i X$ is the class built up from $j$-indexed numbers just the way $X$ is built up from unindexed numbers. 

A \emph{sequence of $i$-th order classes} for $1< i\leq n+2$, is an $i$-th order class with every element $j$-indexed for some $j$.   The $j$-component of any class $x$ is the class of all $y$ with $j\star_i y\in x$. 

\begin{definition} For any class $\alpha$  and sequence $\sigma$ of classes, all of order $i>1$, write $\langle\alpha \rangle$ for the sequence with 0 component $\alpha$ and all others empty; and $\langle\alpha \rangle^\smallfrown\sigma$ for the sequence with $\alpha$ as $0$ component and each $m$  component re-indexed by $m+1$. 
\end{definition}

\begin{definition}A \emph{finite sequence of non-empty classes}, of length $l\in \omega$ is a sequence of classes with $k$-th component nonempty for $k< l$, and empty for $k\geq l$.
\end{definition}  

So a $0$ length sequence of nonempty classes is an empty class.

\section{ Interpreting sets as trees}
\subsection{Generalities on trees}

We interpret sets by trees where each node codes a set with its daughters as elements.  A tree is a class of finite sequences of nonempty classes such that each initial segment of a sequence in the class is also in it.  The daughters of a sequence are the sequences extending it by one entry.  

For example, for distinct nonempty order $n+1$ classes $\alpha,\beta,\gamma$,  the $n+2$ class of sequences $ \{ \langle\rangle, \langle\alpha\rangle, \langle\beta\rangle, \langle\beta,\alpha\rangle, \langle\gamma\rangle, \langle\gamma,\alpha\rangle, \langle\gamma,\beta\rangle, \langle\gamma,\beta,\alpha\rangle \}$ with $\langle\rangle$ empty, is a  tree:
  
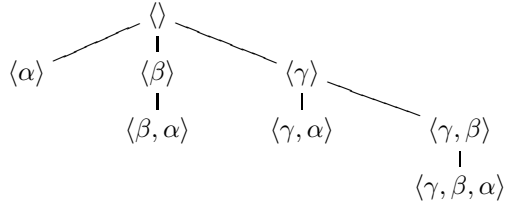
\[ \xymatrix@R=.2cm{ &  \langle\rangle \ar@{-}[dl] \ar@{-}[d]   \ar@{-}[dr] \\   
                                   \langle\alpha\rangle & \langle\beta\rangle \ar@{-}[d]  & \langle\gamma\rangle \ar@{-}[d] \ar@{-}[dr] \\  
                   & \langle\beta,\alpha\rangle & \langle\gamma,\alpha\rangle & \langle\gamma,\beta\rangle \ar@{-}[d] \\
                                       &&& \langle\gamma,\beta,\alpha\rangle } \]
Three nodes below $\langle\rangle$ show this codes a three element set.  No nodes below the leftmost node $\alpha$ shows this node codes the empty set.  Altogether this tree encodes the von~Neumann ordinal $2$, that is $ \{ \emptyset, \{\emptyset\},\{\emptyset,\{\emptyset\}\} \}$.  In this way the class of all strictly descending sequences of numbers codes the ordinal $\omega$.

Tree relations $\in^*,=^*$ representing membership and equality of sets are precisely defined at \citep[pp.264--65]{SimpsonBook}.

\subsection{$Z_2$ interprets $\mathrm{ZF}[0]+(\mathrm{V=L})$}\label{S:secondorder} This sets the stage for our general proof.   The set theory $\Pi ^1_{\infty}\text{-}\mathsf{CA}_0^{\mathrm{set}}$\!  \citep[p.~284]{SimpsonBook} is our $\mathrm{ZF}[0]$ without replacement but with comprehension and hereditary countability.  The axioms are:

\begin{itemize}
    \item Extensionality, empty set, regularity, pair set, sum set, and infinity
    \item Unrestricted comprehension.  For any formula $\phi$ in set theoretic language: 
                          \[ \   \forall u\ \exists v\ \forall x\ ( x\in v\ \leftrightarrow\ (x\in u \land \phi(x) ) \]
    \item Hereditary countability (every set lies in a countable transitive set): 
                \[  \forall u\ \exists v\ (u\subseteq v \land \mathrm{Trans}(v) \land \exists \text{ one-to-one }
                                             \marrow{g}{v}{\omega}     )    \]
\end{itemize}

\begin{definition} \quad
\begin{itemize}
    \item \emph{(In  $\Pi ^1_{\infty}\text{-}\mathsf{CA}_0^{\mathrm{set}}$)}  A \emph{suitable tree} is a set $T$\! of finite 
                     sequences of elements of  $\omega$, where $T$\! is closed under initial segments and has no path. 
                       I.e.~it has no infinite chain of sequences each daughter to the one before.
  \item \emph{(In $Z_{2}$)}  A \emph{suitable tree} is an order $2$ class $T$\! of finite sequences of numbers closed under
        initial segments and having no path.
\end{itemize}
\end{definition} 

\begin{theorem}\label{T:biinterpretation} $\Pi ^1_{\infty}\text{-}\mathsf{CA}_0^{\mathrm{set}}$\! is a conservative extension of $Z_2$\! when we interpret numbers in $Z_2$\! as  elements of $\omega$ and order 2 classes as subsets of $\omega$.
\end{theorem}
\begin{proof} Each theory proves every statement is equivalent to a statement about trees, and the two theories prove all the same statements on trees.   Simpson's  Theorem VII.3.34, sums up a long series of proofs.
\end{proof}

This interpretation of set theory will not provably satisfy replacement, since all trees in $Z_2$ code countable sets and even $\mathrm{ZF}$ does not prove every countable family of countable sets has a countable union.  So we go to constructibility.  \citet[\S VII.4]{SimpsonBook} shows a statement $u\in \mathrm{L}^{\omega}$\! in $\Pi ^1_{\infty}\text{-}\mathsf{CA}_0^{\mathrm{set}}$\! says $u$ is constructed from natural number parameters by some ordinal.  The \emph{constructibility interpretation} $\mathrm{L}^{\omega}$\! is given in $\Pi ^1_{\infty}\text{-}\mathsf{CA}_0^{\mathrm{set}}$\! by relativizing quantifiers to $\mathrm{L}^{\omega}$\!.  This interpretation has a definable well ordering but still suffers a lack of control over countability.   So within $\mathrm{L}^{\omega}$\! Simpson defines $\mathrm{HCL}(\emptyset)$, the \emph{hereditarily constructibly countable} sets. 

\begin{definition} \emph{(In $\Pi ^1_{\infty}\text{-}\mathsf{CA}_0^{\mathrm{set}}$\!)}
 Write $u\in \mathrm{HCL}(\emptyset)$\! to say there is a constructible surjection \eparrow{f}{\omega}{T}\! onto a transitive set $T$\! with $u\subseteq T$\!.
\[  \exists\, f\in  \mathrm{L}^{\omega}\ (\, \mathrm{Fcn}(f)\ \land\ 
                   \mathrm{dom}(f)=\omega  \land\  
                   u\subseteq \mathrm{rng}(f)\ \land\  \mathrm{Trans}({rng}(f))\, )\]
\end{definition} 

\begin{theorem}\label{T:hcl} The $\mathrm{HCL}(\emptyset)$\! interpretation in $\Pi ^1_{\infty}\text{-}\mathsf{CA}_0^{\mathrm{set}}$\! satisfies  $\mathrm{ZF}[0]+(\mathrm{V=L})$\! plus hereditary countability.  
\end{theorem}
\begin{proof} Simpson Theorem~VII.5.4.  The proof shows $\mathrm{HCL}(\emptyset)$\! satisfies a choice principle stronger than replacement. For every $\phi(x,y)$ with $f$\! not free: 
 \[ \forall x\, \exists y\ \phi(x,y)\ \rightarrow\ \forall u\, \exists f\, \forall x\ (\, x\in u \rightarrow \phi(x,f(x))\,) \qedhere \]
\end{proof}

Looking towards the proof of Theorem~\ref{T:penultimate}, note the sets in $\mathrm{L}^{\omega}$\! are countable but need not all be constructibly countable.  If we interpret this whole construction in  $\mathrm{ZF}+(\mathrm{V=L})$ then all sets in $\mathrm{L}^{\omega}$\! are constructibly countable.  At the other extreme, by interpreting this construction in models of $\mathrm{ZF}$ with cardinal collapse, $\mathrm{L}^{\omega}$\! can include $\aleph_{k}^{\mathrm{L}}$ for any finite $k$ and indeed $\aleph_{\alpha}^{\mathrm{L}}$ for any ordinal $\alpha$.

\section{Constructibility in $\Pi ^{\Omega}_{\infty}\text{-}\mathsf{CA}_0^{\mathrm{set}}$\!}\label{S:heartofitall}
The set theory $\Pi ^{\Omega}_{\infty}\text{-}\mathsf{CA}_0^{\mathrm{set}}$\! is $\Pi ^1_{\infty}\text{-}\mathsf{CA}_0^{\mathrm{set}}$ but with countability replaced by a constant $\Omega$ for an indeterminate ordinal in which all sets hereditarily embed:
\begin{itemize}
    \item $\Omega$ is an ordinal.
    \item $\forall u\ \exists v\ (\, u\subseteq v \land\mathrm{Trans}(v) \land  \exists \text{ one-to-one } \marrow{g}{v}{\Omega}\, )$
\end{itemize}
Yhe case $u=\omega$ implies $\omega\leq \Omega$.   The case $u=\cprod{\Omega}{\Omega}$ implies there are one-to-one pairing functions \clmarrow{\cprod{\Omega}{\Omega}}{\Omega}.  Pick one to write as \marrow{\langle\_,\_\rangle}{\cprod{\Omega}{\Omega}}{\Omega}.  For any nonempty transitive set $u$ any one-to-one function \marrow{g}{u}{\aleph_{n}} will serve to define \emph{G\"odel ordinals} coding the language of set theory augmented by a constant $\underline{a}$ for each $a\in U$.  We work in $\Pi ^{\Omega}_{\infty}\text{-}\mathsf{CA}_0^{\mathrm{set}}$\!.  

For $i\in\omega$ the pair $\langle 0,i\rangle$\! codes variable $v_i$.  For $a\in u$, $\langle 1,g(a)\rangle$ codes the constant $\underline{a}$.  Variables and constants are terms.  When $\sigma,\tau$ code terms then $\langle 2,\langle \sigma,\tau\rangle\rangle$ and  $\langle 3,\langle \sigma,\tau\rangle\rangle$ code formulas $\sigma=\tau$ and $\sigma\in\tau$ respectively.  When $\phi,\psi$ code formulas, $\langle 4,\phi\rangle$ and $\langle 5,\langle \phi,\psi\rangle\rangle$ code $\neg\phi$ and $\phi\land \psi$ and $\langle 6,\langle i,\phi\rangle\rangle$ codes $\forall x_i\phi$.  The methods of  \citet[\S VII.4]{SimpsonBook} work as well in $\Pi ^{\Omega}_{\infty}\text{-}\mathsf{CA}_0^{\mathrm{set}}$\! using these codes to define constructiblity.  Write  $x\in \mathrm{L}_{\alpha}^{\omega}$ to say $x$ is constructed from parameters in $\omega$ by ordinal $\alpha$. 

The proof of Theorem~\ref{T:penultimate} rests on the interplay of two senses of cardinality, which we state precisely.  We define cardinals as well ordered cardinals:

\begin{definition} \emph{(In $\Pi ^{\Omega}_{\infty}\text{-}\mathsf{CA}_0^{\mathrm{set}}$)} 
An ordinal $\alpha$ is a \emph{cardinal} if no ordinal $\beta<\alpha$ admits a surjective function \eparrow{q}{\beta}{\alpha}, and is a \emph{constructible cardinal}  if no ordinal $\beta<\alpha$ admits a constructible surjective function \eparrow{q}{\beta}{\alpha}.
 \end{definition}

\begin{definition} \emph{(In $\Pi ^{\Omega}_{\infty}\text{-}\mathsf{CA}_0^{\mathrm{set}}$)} 
For any ordinals $\alpha,\beta$ write \clmarrow{\alpha}{\beta} to say $\alpha$ has some one-to-one function to $\beta$. For any ordinal $\Theta\leq\Omega$ write  $u\in \mathrm{H\Theta L}$\! to say  $u$ is \emph{hereditarily $\Theta$ constructible} from parameters in $\omega$:
  \[  \exists \,   \clmarrow{\alpha}{\Theta}\ \exists\, f\in  \mathrm{L}^{\omega}_{\alpha}\, ( \mathrm{Fcn}(f)\, \land\, 
                   \mathrm{dom}(f)=\Theta\,  \land\,  
                   u\subseteq \mathrm{rng}(f)\, \land\,  \mathrm{Trans}({rng}(f)) )\]
 \end{definition}

\begin{theorem}\label{T:penultimate}\emph{(In $\Pi ^{\Omega}_{\infty}\text{-}\mathsf{CA}_0^{\mathrm{set}}$)} \quad
 \begin{enumerate}
     \item The $\mathrm{H}\Omega\mathrm{L}$ sets satisfy $\Pi ^{\Omega}_{\infty}\text{-}\mathsf{CA}_0^{\mathrm{set}}$\! plus $\mathrm{V=L}$.
     \item  If constructible cardinal $\aleph_n^{\mathrm{L}}$ exists, the $\mathrm{H\aleph_n^{\mathrm{L}}}$ sets satisfy $\mathrm{ZFG}[n^+]$.
\end{enumerate}
\end{theorem}
\begin{proof}   Simpson's $\mathrm{HCL}$ is our $\mathrm{H\aleph_0^{\mathrm{L}}}$.    Simpson's Theorem~VII.5.4 adapts to show part 1 using ordinals \clmarrow{\alpha}{\Omega} in place of countable ordinals.  Part 2 is immediate. 
\end{proof}

\section{$Z_{n+2}$ interprets $\mathrm{ZFG}[n^+]+(\mathrm{V=L})$}
\begin{definition} \emph{(In $Z_{n+2}$)} 
\begin{itemize}
 \item  For any class $C$\! a \emph{suitable $C$-tree}\! is a class $T$  of finite sequences of elements of $C$\!,
              closed under initial segments and having no path. 
 \item The \emph{canonical tree} $C^*$ on any well ordered class $C$ is the class of strictly descending sequences of 
                         elements of $C$.   So it is a suitable $C$-tree.
\end{itemize}
\end{definition}

\begin{theorem} \emph{(In  $Z_{n+2}$)}  For any definably well ordered class $C$ the $C$-suitable trees satisfy $\Pi ^{\Omega}_{\infty}\text{-}\mathsf{CA}_0^{\mathrm{set}}$\!.
\end{theorem}
\begin{proof} The proof of Simpson's Theorem 3.33 adapts to any fixed well order replacing $\omega$.  Code the sets $\omega$ and $\Omega$ by the canonical trees for the well ordered classes $\omega,C$.
\end{proof}

By Theorem~\ref{T:penultimate} this paper is done when we show:

\begin{theorem}  $Z_{n+2}$ proves there is a definably well ordered class $C$\! such that the suitable $C$-tree
              interpretation includes $\aleph_n^{\mathrm{L}}$.
\end{theorem}

\begin{proof} By Theorem~\ref{T:hcl}, the order 2 class $C^0=\omega$ gives a constructibility interpretation including $\aleph_0$ and possibly larger cardinals.  If it has at least $n$ successive constructible cardinals $\aleph_0,\dots,\aleph^\mathrm{L}_n$ the proof is done.  

Otherwise $\mathrm{L}^{\omega}$ has a largest constructible ordinal $\aleph^\mathrm{L}_i$ (in fact $i<n$ but we do not use that).  So each ordinal in $\mathrm{L}^{\omega}$ is constructibly embedded in  $\aleph_i^{\mathrm{L}}$ and so $\mathrm{H\aleph_i^{\mathrm{L}}}=\mathrm{L}^{\omega}$\!.  By Lemma~\ref{L:stepup} there is a definably well ordered order $3$ class of selected trees coding representatives of all ordinals in $\mathrm{L}^{\omega}$\!.   Call this class $C^1$\!.  

Form the suitable $C^1$-tree interpretation, and the constructibility interpretation within it.  Call that constructibility interpretation $\mathrm{L}^{\omega,1}$\!.  Like $\mathrm{L}^{\omega}$ it starts with set parameters in $\omega$, but its ordinals are given by suitable $C^1$-trees rather than suitable $\omega$-trees.  It includes intuitively all ordinals of cardinality $C^1$\!.  By Lemma~\ref{L:noncollapse} there is no surjective function between the classes \cleparrow{\omega}{C^1}\!.  Quite apart from issues of constructibility, $Z_{n+2}$ proves there is no such surjection.  Since the canonical tree $C^{1*}$\! codes an ordinal, it follows that $\mathrm{L}^{\omega,1}$\! includes at least a cardinal $\aleph_1$\! (not merely constructible $\aleph_1^{\mathrm{L}}$).

If $\mathrm{L}^{\omega,1}$\! includes at least $\aleph^\mathrm{L}_n$\! the proof is done.  Otherwise $\mathrm{L}^{\omega,1}$\! has a largest constructible ordinal $\aleph^\mathrm{L}_i$ and $\mathrm{H\aleph_i^{\mathrm{L}}}=\mathrm{L}^{\omega,1}$\!.   Lemmas~\ref{L:stepup} and~\ref{L:noncollapse} show there is a definably well ordered order $4$ class of selected trees coding all ordinals in $\mathrm{L}^{\omega,1}$\!, call it $C^2$\!, it naturally includes a copy of $C^1$ as initial segment, and there is no surjective function \cleparrow{C^1}{C^2}\!.  The corresponding constructibility interpretation $\mathrm{L}^{\omega,2}$\! extends $\mathrm{L}^{\omega,1}$\! and includes at least a cardinal $\aleph_2$.

As long as we do not reach $\aleph^\mathrm{L}_{n}$ we get new well ordered classes $C^{k+1}$ each of order one higher than the preceding $C^k$\! and with an initial segment copying it.  So $\mathrm{L}^{\omega,k+1}$\! includes a copy of $\mathrm{L}^{\omega,k}$\!.   None of these classes $C^k$\! admits any surjection from any earlier one, so their canonical trees represent distinct cardinals in $\mathrm{L}^{\omega,k+1}$\!.   By $k=n$ at the latest we reach $\aleph^\mathrm{L}_n$.
\end{proof}

\begin{definition} A \emph{membership determined} subtree of a tree $T$ is a subclass $S\subset T$ such that for every sequence $\sigma$ in $T$, $\sigma$ is in $S$ iff its first entry is in $S$.
\end{definition}

So, in the suitable $C$-tree interpretation of set theory for any well ordered set $C$, every subset of the set coded by the canonical tree on $C$ is coded by a unique membership determined subtree of that canonical tree.  

\begin{lemma}\label{L:stepup} \emph{(In $Z_{n+2}$)}  Suppose the constructible suitable tree interpretation $\mathrm{L}^{\omega,C}$ on a definably well ordered class $C$\! of order $i<n+2$ includes a constructible cardinal $\aleph_{\alpha}^{\mathrm{L}}$ with $\mathrm{H\aleph_{\alpha}^{\mathrm{L}}}=\mathrm{L}^{\omega,C}$\!.  Then there is a definably well ordered class $C'$\! of order $i+1$ containing one representative for each ordinal in $\mathrm{L}^{\omega,C}$\!.
\end{lemma}

\begin{proof} Pick a pairing function \marrow{\langle\_,\_\rangle}{\cprod{C^*}{C^*}}{C^*} in the set theoretic interpretation.  Each ordinal $\alpha$ in $\mathrm{L}^{\omega,C}$\! is isomorphic to many different relations on $C^*$ (interpreted as a set) and thus to many subsets of $C^*$, but one of these subsets comes first in the constructibility order.  Call that subset $S_{\alpha}$.  Then $\alpha$ as a set in the $\mathrm{L}^{\omega,C}$\!  interpretation is uniquely determined by the canonical subtree coding $S_{\alpha}$.  Call that tree $\overline{\alpha}$ the chosen representative of $\alpha$.

So define $C'$ as the order 3 class of all finite sequences $\langle\overline{\alpha} \rangle^\smallfrown\psi$ for $\alpha$ an ordinal in  $\mathrm{L}^{\omega,C}$\! and $\psi$ the singleton image of a sequence $\sigma\in\overline{\alpha}$.  That is, the successive entries in $\psi$ are the order 2 singleton classes of the successive order 1 entries in $\sigma$.   The order relation on ordinals in $\mathrm{L}^{\omega,C}$\! gives a definable well order on $C'$\!. 
\end{proof}

\begin{lemma}\label{L:noncollapse}  In the situation of Lemma~\ref{L:stepup} there is no surjective function \cleparrow{C}{C'}\!.
\end{lemma}
\begin{proof}   The canonical tree $C'^*$\! on $C'$\! codes an ordinal such that $\alpha\in C'^*$\! for each ordinal in the  suitable $C$-tree interpretation.   So the tree equality relation $=^*$ cannot hold between $C'^*$ and any $C$ tree.  But if there is a surjective \eparrow{q}{C}{C'} between the classes then take the suitable $C$-tree $Q^*$\! containing all sequences $\langle\beta \rangle^\smallfrown\sigma$  such that $q(\beta)=\overline{\alpha}$ for some $\overline{\alpha}\in C'$ and  $\sigma\in\overline{\alpha}$.  Then $Q^*=^*C'^*$ by the isometry relating each $\langle\beta \rangle^\smallfrown\sigma\in Q^*$ to $\langle\overline{\alpha} \rangle^\smallfrown\psi$ where $q(\beta)=\overline{\alpha}$ and $\psi$ is the singleton image of  $\sigma$. 
\end{proof}


\end{document}